\documentclass[a4paper,3pt]{article}

\usepackage{amsmath,amssymb,amsfonts,amsthm}
\usepackage{color}
\usepackage{pdfpages}
\usepackage{graphicx}
\usepackage{geometry}
\usepackage{url}
\usepackage{enumerate}
\usepackage{extarrows}
\usepackage{extarrows,booktabs,multirow,diagbox}

\newtheorem{thm}{Theorem}
\newtheorem{remark}{Remark}
\newtheorem{lem}{Lemma}[section]

\theoremstyle{definition}
\newtheorem{defn}{Definition}

\newcommand{\Rg}       {{\hbox{I\kern-.22em\hbox{R}}}}
\newcommand{\Pg}       {{\hbox{I\kern-.22em\hbox{P}}}}
\newcommand{\Eg}       {{\hbox{I\kern-.22em\hbox{E}}}}

\usepackage{color}
\newcommand{\red}{\color[rgb]{0.8,0,0}}

\title{Mixed Sub-fractional Brownian Motion and Drift Estimation of Related Ornstein-Uhlenbeck Process}

\author{
Chunhao Cai\\
School of Mathematics, Shanghai University of Finance and Economics, Shanghai, China. \\
Email: caichunhao@mail.shufe.edu.cn\\
\and 
*Qinghua Wang\\
Corresponding author\\
School of Mathematics, Shanghai University of Finance and Economics, Shanghai, China. \\
Email: wangqinghua@mail.shufe.edu.cn\\
\and 
Weilin Xiao\\
School of Management, Zhejiang University, Hangzhou, China. \\
Email: wlxiao@zju.edu.cn.}
\date{\today}

\begin{document}
\maketitle

\begin{abstract}
In this paper, we will first give the numerical simulation of the sub-fractional Brownian motion through the relation of fractional Brownian motion instead of its representation of random walk. In order to verify the rationality of this simulation, we propose a practical estimator associated with the LSE of the drift parameter of mixed sub-fractional Ornstein-Uhlenbeck process, and illustrate the asymptotical properties according to our method of simulation when the Hurst parameter $H>1/2$. \\

\vspace{2mm}


\vspace{2mm}


\end{abstract}
\section{Introduction}
Temporal dependence in the volatility of financial assets has been one of the most interesting problems in financial economics. For example, Gatheral et \textit{a.l.} \cite{GJR18} introduced an important phenomenon in finance: the volatility is rough. Furthermore, Euch, Fukasawa and Rosenbaum (\cite{ER18}, \cite{ER19}, \cite{EFR18}) studied so many models under the rough volatility. In these works, they proved that the scaling limit of some nearly unstable Hawkes process is a rough model with the key formula $\int_0^t (t-s)^{H-1/2}dW_s$, where $W_s$ is a standard Brownian motion.

As we know, the sub-fractional Brownian motion (sfBm) arising from the occupation time fluctuations of branching particle systems with Poisson initial condition also presents the properties of long-range dependence and the rough dependence for different Hurst parameter. It will also be a potential candidate to model noise in mathematical finance (see \textit{e,g,} \cite{liu2010sub}) even if its  increment is not stationary.

However, as far as we know, there are only a few works concerning on the numerical simulation of the sfBm. We find the method of the random walk based on the convergence in distribution in \cite{MF15}, but sometimes this method is not so accurate. This brings us a nature idea,  can we find the simulation in the sense of strong convergence? We will find the answer in Section \ref{preliminaries} through the relations between sfBm and fBm. In order to verify the rationality of this simulation, we will propose a practical estimator associated with the LSE of the drift parameter of mixed sub-fractional Ornstein-Uhlenbeck process (msfOU for short), and utilize our numerical results to illustrate the asymptotical properties when the Hurst parameter $H>1/2$.

Why the  mixed sub-fractional Brownian motion (msfBm for short)? In fact, a pure sfBm is just an extension of the fractional Brownian motion and we have almost the same results without extra complicated calculations, but when we add the sfBm with an independent Brownian motion (so called msfBm), the properties change a lot such as the quadratic variation, the stochastic integral, the semi-martingale representation according to \cite{CCK}. In the non-ergodic Ornstein-Uhlenbeck case, these changes will be more clear because it means that we can not directly use the Young integral presented in \cite{el2016least}. Here we  want to find how these changes bring us the differences in finding the asymptotical properties of the LSE of the drift parameter of msfOU (the Lemma \ref{s s distance} for example and others).  

In order to achieve the mentioned goals, in this paper we will define the Malliavin derivative and its adjoint operator (or the Skorohod integral) with respect to the msfBm when $H>1/2$ and try to find out the relation between them and that with respect to the standard Brownian motion with the method of fundamental martingale in \cite{CCK}. At the same time we also find the LSE of the drift parameter of msfOU process through the Skorohod integral and demonstrate the asymptotical properties not only in the ergodic case but also in the non-ergodic case.
\begin{remark}
With the method of fundamental martingale, the msfBm can be presented by the stochastic integral $\int_0^t g(s,t)dW_s$. Similar to the rough model in \cite{ER18}, \cite{ER19}, \cite{EFR18}, we can try to find the conditions of one type unstable Hawkes process whose scaling limit is this process and it will be our future works.
\end{remark}

The rest of this paper is organized as follows. In Section \ref{preliminaries} we will introduce our method of simulation and the Skorohod integral as well as the path-wise integral with respect to the msfBm.  The exact formula of the practical estimator and its asymptotic properties will be discussed in Section \ref{sec: estimation}.  For the sake of completeness, the non-ergodic case of the LSE of the msfOU process is discussed in Section \ref{sec: non ergodic}. Section \ref{sec: simulation} is devoted to presenting Monte Carlo studies on the finite sample properties of the practical estimator with $1/2<H<3/4$. Some technical proofs are collected in the Appendix.

\section{Preliminaries}\label{preliminaries}

\subsection{sub-fractional Brownian motion and simulation}


Let  $(\Omega, \mathcal{F}, \mathcal{F}_t, \mathbf{P})$ be a filtered probability space, the sub-fractional Brownian motion $S^{H}=\{S^H_t, t \geq 0\}$  with an initial $S^{H}_0=0$ and the index $H\in(0,1)$ is a mean zero Gaussian process defined by the covariance function:
\begin{equation}\label{eq: cov St}
R_H(s,t)=\mathbf{E}(S_t^H S_s^H)=t^{2H}+s^{2H}-\frac{1}{2}\left(|t-s|^{2H}+|t+s|^{2H}\right),\, s,t\in [0,T].
\end{equation}
As we know, for $H = 1/2$, $S^{H}$ coincides with the standard Brownian motion. In fact, $S^{H}$ is
neither a semimartingale nor a Markov process for other $H$ and specially for  $H>1/2$, we can see that
$$\mathbf{E}S_t^HS_s^H=\int_0^t\int_0^sK_H(r,u)drdu,\, 0\leq s,t\leq T$$
with
\begin{equation}\label{K H}
K_H(s,t)=\frac{\partial ^2}{\partial s\partial t}R_H(s,t)=H(2H-1)\left(|s-t|^{2H-2}-(s+t)^{2H-2}\right).
\end{equation}

In \cite{zili13}, the author presented a very important relationship between sfBm and fBm, that is for any $H\in (0,1)$
\begin{equation}\label{eq: sum fBm}
S_t^H=\frac{1}{\sqrt{2}}(B_t^H+B_{-t}^H),\, H\in (0,1),\, 0\leq t \leq T
\end{equation}
where $B^H=(B_t^H)_{-\infty < t< \infty}$ is fBm with Hurst index $H$ on the whole real line. We make this equation \eqref{eq: sum fBm} as the idea of the simulation of sub-fractional Brownian motion using Paxson's algorithm (see, \cite{paxson1997}). The procedure follows:



\begin{itemize}

\item  Set the sampling size $N$ and the time span $T$ and obtain the sampling interval by $d=T/N$;

\item Set the values of two variables $H$ and generate fBm $B_{2T}^H$ based on Paxson's method (see \cite{paxson1997}), with the sampling interval $d$ and $2N$ points. Thus, we have
$0=t_0<t_1<t_2<\cdots<t_{2N}=T,\, t_i-t_{i-1}=d$;

\item  consider the sequence of $\bf{Y}$ with $2N$ dimension
$$Y_1=B_{t_1}^H, Y_2=B_{t_2}^H-B_{t_1}^H, \cdots, Y_i=B_{t_i}^H-B_{t_{i-1}}^H,\cdots, Y_{2N}=B_{t_2N}^H-B_{t_{2N-1}}^H;$$

\item  Construct a new real line fBm denoted $\tilde{B}^H$:
$$\tilde{B}_{t_i}^H=\sum_{j=1}^i Y_{N+j},\, i=1,2,\cdots,N,\:\;\;\tilde{B}_{-t_i}^H=-\sum_{j=1}^i Y_{N-(j-1)},\, i=1,2,\cdots,N;$$

\item   Use $S_{id}^H=\frac{1}{\sqrt{2}}(\tilde{B}_{t_i}^H+\tilde{B}_{-t_i}^H)$ to obtain sfBm;

\end{itemize}
The following two figures represent the sfBm with $H=0.7$ and $H=0.3$

\begin{center}
\resizebox{100mm}{60mm}{\includegraphics{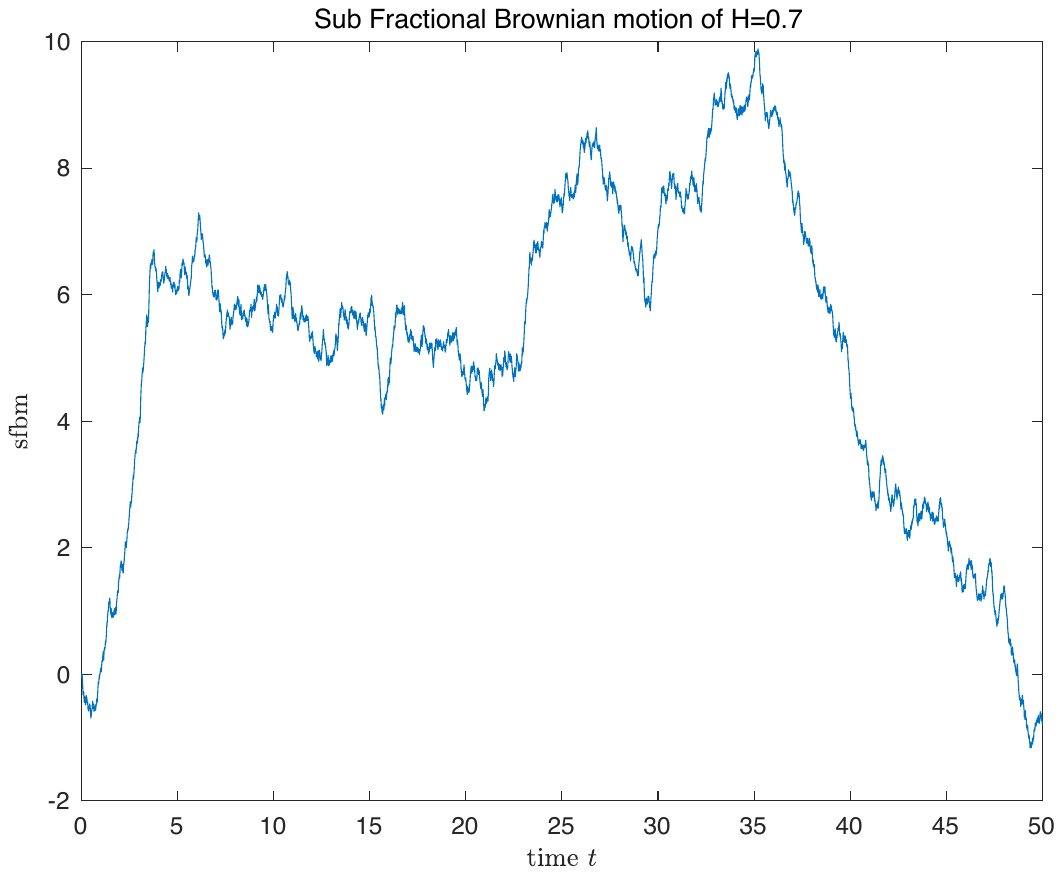}}\\
\resizebox{100mm}{60mm}{\includegraphics{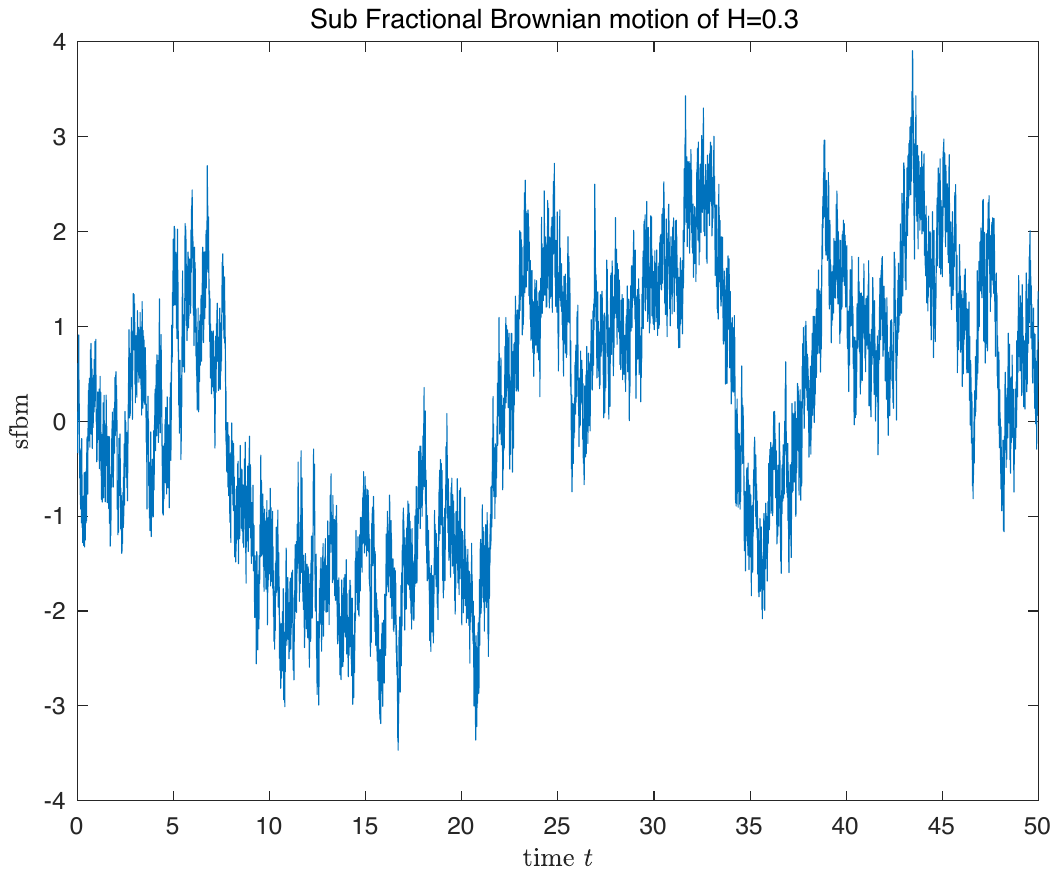}}\\
\end{center}

\subsection{Malliavin derivative and adjoint operator with respect to msfBm}

From now on, we only consider our model for $H>1/2$. Following the idea of \cite{CCK}, we introduce the process of mixed sub-fractional Brownian motion $\xi=(\xi_t,\,0\leq t\leq T)$ which is defined by
\begin{equation}\label{mixed def}
\xi_t=W_t+S_t^H,
\end{equation}
where $W=(W_t,\,0\leq t\leq T)$ is a standard Brownian motion and $S^H=(S_t^H,\, 0\leq t\leq T)$ is an independent sub-fractional Brownian motion. As we know, the stochastic integrals with respect to the standard Brownian motion and fractional Brownian motion are too different when the integrant is a stochastic process, so if we want to define the integral with the process $\xi$, we can not define two integral and then add them together. In this case we have to just consider $\xi$ as a Gaussian process and we use the Malliavin calculus to define this type stochastic integral.

We consider a fixed interval $[0,T]$ and denote $\mathcal{E}$ be the set of step function on $[0,T]$. Let $\mathcal{H}$ be the Hilbert space defined as the closure of $\mathcal{E}$ with respect to the scalar product
$$
\langle \mathbf{1}_{[0,t]},\, \mathbf{1}_{[0,s]}\rangle_{\mathcal{H}}=R_H(t,s)+t\wedge s
$$
where $R_H(t,s)$ defined in \eqref{eq: cov St}. As presented in Definition 1.1.1 of \cite{Nualart}, we can define an isometry from the Hilbert space $\mathcal{H}$ to the Gaussian space $\mathcal{H}_1$ associated with $\xi$ as the extension of the mapping $\mathbf{1}_{[0,t]}\rightarrow \xi_t$. This isometry will be presented by $\varphi\rightarrow \xi(\varphi)$ for every function $\varphi\in \mathcal{H}$. Now, for any pair step functions $f,g\in \mathcal{H}$ and $H>1/2$, we have
\begin{equation}\label{fg}
\langle f,\,g\rangle_{\mathcal{H}}=\int_0^T f(t)g(t)dt+\alpha_H\int_0^T\int_{0}^{T}f(t)g(s)\left(|t-s|^{2H-2}-(t+s)^{2H-2}\right)dsdt\,,
\end{equation}
where $\alpha_H=H(2H-1)$. Following the same steps as Section 1.2 in \cite{Nualart}, let  $F\in \mathcal{S}$  the class of smooth random variables with the formula
\begin{equation}\label{smooth random}
F=f(\xi(h_1),\,\dots,\, \xi(h_n)),
\end{equation}
where $f\in C_b^{\infty}(\mathbb{R}^n)$, $h_1,\,\dots,\, h_n \in\mathcal{H}$ for $n\geq 1$.  The Malliavin derivative $D_{\xi}F$  with respect to $F$ satisfying the chain rule,  is provided by the following definition.
\begin{defn}
The Malliavin derivative of $D_{\xi}F$ of the smooth random variable $F$ with the exact formula \eqref{smooth random} is a $\mathcal{H}-$valued random variable given by
\begin{equation}\label{mallivanderivative}
D_{\xi}F=\sum_{i=1}^n \frac{\partial f}{\partial x_i}\left(\xi\left(h_1\right),\,\dots,\, \xi\left(h_n\right)\right)h_{i}\,.
\end{equation}
\end{defn}
Now, let us define the space $\mathbb{D}_{\xi}^{1,2}$ which is the closure of the class of of smooth randoms variables $\mathcal{S}$ with the norm
$$
||F||_{1,2}=\left(\mathbf{E}(|F|^2)+\mathbf{E}||D_{\xi}F  ||_{\mathcal{H}}^2 \right)^{1/2},\, F\in \mathbb{D}_{\xi}^{1,2},
$$
then it is obvious that $\mathbb{D}_{\xi}^{1,2}$ is a Hilbert space,  we can define the adjoint of the operator $D_{\xi}$
\begin{defn}
Let $\delta_{\xi}$ be the adjoint of the operator $D_{\xi}$. Then $\delta_{\xi}$ is an unbounded operator on the space $L^2(\Omega;\mathcal{H})$ with values in $L^2(\Omega)$ such that
\begin{itemize}
\item  The domain of $\delta_{\xi}$, denoted by $\emph{Dom}(\delta_{\xi})$, is the set of $\mathcal{H}-$valued square integrable random variables $u\in L^2(\Omega;\mathcal{H})$ such that for any $F\in \mathbb{D}_{\xi}^{1,2}$,$$
|\mathbf{E}(\langle D_{\xi}F,\,u   \rangle_{\mathcal{H}})|\leq c||F||_{2}\,.
$$
where $c$ is a constant depending on $u$.

\item If $u$ belongs to $\emph{Dom}(\delta_{\xi})$, then $\delta_{\xi}(u)$ is the element of $L^2(\Omega)$ characterized by
$$
\mathbf{E}(F\delta_{\xi}(u))=\mathbf{E}(\langle D_{\xi}F,\,u \rangle_{\mathcal{H}})\,.
$$
\end{itemize}

In fact, if $u\in\emph{Dom}(\delta_{\xi})$, then the Skorohod integral with respect to $\xi$ is $\delta_{\xi}(u)$ and denoted by 
$$
\delta_{\xi}(u)=\int_0^Tu(t)\delta\xi_t\,.
$$
\end{defn}
\begin{remark}
For the deterministic function $\psi(t)\in \mathcal{H}$, it is not hard to check that $\psi(t)\in \emph{Dom}(\delta_{\xi})$, with the same proof as the stochastic calculus with respect to the standard Brownian motion in \cite{FHBZ}, since $\delta_{\xi}(\psi)$ is the Riemmann-Stieltjes intergral.
\end{remark}
If we introduce the process $B=(B_t,\, 0\leq t\leq T)$ be the standard Brownian motion with the same filtration of the process $\xi$, one may ask the following three questions:
\begin{enumerate}
\item  What is the relationship between the Malliavin derivative $D_{\xi}$ and $D_{B}$
\item  What is the relationship between the adjoint operator $\delta_{\xi}$ and $\delta_{B}$?
\item What is the operator of $\mathcal{G}^*$ such as the $K_H^*$ presented in Chapter 5 of \cite{Nualart}?
\end{enumerate}
To answer these questions, we first introduce the fundamental martingale of the msfBm which is also defined in \cite{CCK}. In fact, the fundamental martingale $M=(M_t,\, 0\leq t\leq T)$ and its quadratic variance $\langle M\rangle_t$ are  
\begin{equation}\label{rep:mar}
M_t=\int_0^t g(s,t)d\xi_s,\, \langle M\rangle_t=\int_0^t g(s,t)ds=\int_0^tg^2(s,s)ds,\, t\geq 0,
\end{equation}
where $g(s,t)$ is the solution of the following Wiener-H\"opfner integral equation:
\begin{equation}\label{kernel g}
g(s,t)+\int_0^t g(r,t) \kappa(r,s)dr=1,\,\kappa(s,t)=H(2H-1)(|t-s|^{2H-2}-|t+s|^{2H-2})\,.
\end{equation}
On the other hand, we have the innovation representation immediately from \cite{CCK}:
\begin{equation}\label{eq: innovation}
\xi_t=\int_0^t G(s,t)dM_s,\, t\in [0,T],\,\, G(s,t):=1-\frac{1}{g(s,s)}\int_0^t R(\tau,s)d\tau,\, 0\leq s\leq t \leq T.
\end{equation}
with 
$$
R(s,t):=\frac{\dot{g}(s,t)}{g(t,t)},\, s\neq t,\,\, \dot{g}(s,t):=\frac{\partial}{\partial t}g(s,t).
$$
Now for a smooth function $\psi(t)\in \mathcal{H}$, it is easy to check:
\begin{equation}\label{eq: test}
\int_0^T \psi(t)d\xi_t=\int_0^T \left[\int_{\tau}^Tg(\tau,\,\tau)\frac{\partial G}{\partial t} (\tau,t)\psi(t)dt+\psi(\tau)\right]dB_{\tau}
\end{equation}
and we can define an operator $\mathcal{G}^*$ from $\mathcal{H}$ to the complete subspace $L^2[0,T]$:
$$
(\mathcal{G}^*\psi)(\tau)=\int_{\tau}^T\frac{\partial \mathcal{G}}{\partial t} (\tau,t)\psi(t)dt+\psi(\tau),\, \,   \frac{\partial \mathcal{G}}{\partial t} (\tau,t)=g(\tau,\,\tau)\frac{\partial G}{\partial t} (\tau,t).
$$
the divergence-type integral with respect to msfBm will be defined immediately:


\begin{defn}
Let $u$ be a stochastic process such that for every trajectory it is a mapping from the interval $[0,T]$ to $\mathcal{H}$ and $\mathcal{G}^*u$ is Skorohod integrable with respect to the standard Brownian motion $B_t$. Then we define the extended Wiener integral of $u$ with respect to the mfBm $\xi$ as
\begin{equation}\label{divergence type}
\xi(u):=\int_0^T(\mathcal{G}^*u)(\tau)\delta B_{\tau}.
\end{equation}
\end{defn}
It is easy to check that $\emph{Dom}(\delta_{\xi})=(\mathcal{G}^*)^{-1}(\emph{Dom}(\delta_{B}))$ and for $u\in \emph{Dom}(\delta_{\xi})$ the It\^o-Skorohod integral $\delta_{\xi}(u)$ coincides with the divergence-type integral $\xi(u)$ defined in \eqref{divergence type}. At the same time, from \cite{FHBZ}, we have the following result.

\begin{lem}\label{equivalent derivative}
For any $F\in \mathbb{D}_{B}^{1,2}=\mathbb{D}_{\xi}^{1,2}$, we have
$$
\mathcal{G}^*D_{\xi}F=D_{B}F
$$
where $D_{B}$ denotes the Malliavin derivative with respect to the standard Brownian motion $B$ and $\mathbb{D}_{B}^{1,2}$ the corresponding Sobolev space.
\end{lem}

\begin{remark}
Even in \cite{CCK}, Cai et \textit{a.l.} have found the fundamental martingale of mixed fractional Brownian motion for $H<1/2$, but the corresponding divergence type of the mfBm is still not explicit because in this case we do not have the inter-changeble of the integral and derivative of the kernel $g$. The same problem exists in msfBm and we leave this study for further research.
\end{remark}

\subsection{Path-wise integral with respect to $\xi=(\xi_t,\, 0\leq t \leq T)$}\label{ch: path-wise integral}

Let us put $\xi_s=\xi_T$ for $s>T$ and $\xi_s=0$ for $s<0$. The symmetric path-wise integral of a process $(u_t,\, 0\leq t\leq T)$ with respect to $(\xi_t,\, 0\leq t\leq T)$ is defined by
$$
\int_0^Tu(s)\circ d\xi_s=\lim_{\epsilon \downarrow 0} \frac{1}{2\epsilon}\int_0^Tu(s) \left[\xi(s+\epsilon)-\xi(s-\epsilon)     \right]ds
$$
provided that the limit exists in probability. The following lemma explains the relationship between this symmetric path-wise integral and the Skorohod integral:
\begin{lem}\label{relation}
Suppose that the stochastic process $(u_t)_{t\in [0,T]} \in \emph{Dom}(\delta_{\xi})$ satisfying the following conditions: 
\begin{equation}\label{con1}
\int_0^T\int_0^T|D_s^{\xi}u(t)|\left(|t-s|^{2H-2}-(t+s)^{2H-2}\right)dsdt<\infty,\, a.s, 
\end{equation}
\begin{equation}\label{con3}
D_s^{\xi}u(t)=0,\, s>t
\end{equation}
and
\begin{equation}\label{con2}
\int_0^T |D_t^{\xi}u(t)|dt<\infty,\, a.s
\end{equation}
where $D_t^{\xi}u(t)$ means $D_s^{\xi}u(t)$ when $s=t$.Then the symmetric integral exists and the following relation holds:
\begin{equation}\label{eq:relation}
\int_0^Tu(t)\circ d\xi_t=\delta_{\xi}(u)+H(2H-1)\int_0^T\int_0^T D_s^{\xi}u(t)\left(|t-s|^{2H-2}-(t+s)^{2H+2}\right)dsdt+\frac{1}{2}\int_0^TD_t^{\xi}u(t)dt\,.
\end{equation}
\end{lem}
\begin{remark}
{\red Different from the pure fractional case or sub-fractional case, we have in this Lemma an extra part from the standard Brownian motion $\int_0^TD_t^{\xi}u(t)dt$ which is important in our analysis.}
\end{remark}

\begin{proof}
Let
$$
u_t^{\epsilon}=\frac{1}{2\epsilon}\int_{t-\epsilon}^{t+\epsilon}u(s)ds\,.
$$

Using the integration by part formula, we have
$$
\delta_{\xi}Fu=F\delta_{\xi}(u)-\langle D^{\xi} F, u\rangle_{\mathcal{H}},
$$
for $F\in \mathbb{D}^{1,2}_{\xi}$ and $u\in \emph{Dom}(\delta_{\xi})$.  When $u$ satisfied the conditions \eqref{con1} and \eqref{con2}, we have
 \begin{eqnarray*}
    \int_0^T u(s)\frac{\xi_{s+\epsilon}-\xi_{s-\epsilon}}{2\epsilon}ds &=& \int_0^T u(s)\frac{1}{2\epsilon}\int_{s-\epsilon}^{s+\epsilon}d\xi_uds \\
    &=& \int_0^T\delta_{\xi}\left(u(s)\frac{1}{2\epsilon}\mathbf{1}_{[s-\epsilon,s+\epsilon]}\right)ds+\frac{1}{2\epsilon}\int_0^T\langle D^{\xi}u(s),\mathbf{1}_{[s-\epsilon,s+\epsilon]}\rangle_{\mathcal{H}}ds \\
    &=& \delta_{\xi}(u^{\epsilon})+\frac{1}{2\epsilon}\int_0^T\langle D^{\xi}u(s),\mathbf{1}_{[s-\epsilon,s+\epsilon]}\rangle_{\mathcal{H}}ds\,.
 \end{eqnarray*}

The equation \eqref{eq:relation} can be easily obtained by taking the limit $\epsilon \rightarrow 0$ on both sides of the above equation.
\end{proof}

\section{Estimation of the drift parameter for msfOU process}\label{sec: estimation}
The drift estimations of Ornstein-Uhlenbeck processes with different noises have attracted many interests. It is an important subject in the financial economitreics because one will make a decision which model the volatility satisfies when we can observe the past data. In this part we consider the mixed sub-fractional O-U process (msfOU) 
 $X=(X_t)_{0\leq t \leq T}$ which satisfies the following stochastic differential equation
\begin{equation}\label{msfOU}
dX_t=-\vartheta X_tdt+d\xi_t,\, 0\leq t \leq T.
\end{equation}
with the unknown parameter $\vartheta>0$ is unknown.  Our aim is to estimate this parameter from the continuous observed data $(X_t,\, 0\leq t \leq T)$. The MLE of the mixed fractional O-U process has been considered in \cite{CK18b}, we can also construct the same estimator for the process \eqref{msfOU} but the asymptotical normality is still not clear and this type of estimator is hard to simulate. So  here we still consider the LSE  as presented in \cite{HN10}. The most difficulty or the difference between the sub-fractional case and the fractional case is the stationarity: when sub-fractional Brownian motion is not stationary so with the msfOU process, the strong consistency of the LSE is not immediate. We will write the explicit formula and analyze it with the decomposition $S_t^H=\frac{1}{\sqrt{2}}(B_t^H+B_{-t}^H)$.
\begin{remark}
The non-ergodic case will be considered in Section \ref{sec: non ergodic} because the quadratic variation variation of the mixed case is not 0, we can not use the Young integral and obtain the results directly from \cite{el2016least}.
\end{remark}
\subsection{Least square estimator}
From \cite{HN10}, we can easily obtain the Least Square Estimator 
\begin{equation}\label{eq:least square}
\bar{\vartheta}_T=-\frac{\int_0^TX_tdX_t}{\int_0^TX_t^2dt}=\vartheta-\frac{\int_0^TX_td\xi_t}{\int_0^TX_t^2dt}\,,
\end{equation}
where the stochastic integral $\int_0^TX_td\xi_t$ is interpreted as the Skorohod integral and it will be denoted as $\int_0^TX_t\delta \xi_t$ and we have the following result:
\begin{lem}
For $H>1/2$, the LSE
\begin{equation}\label{eq3: least square}
\bar{\vartheta}_T=-\frac{X_T^2}{2\int_0^T X_t^2dt}+\frac{\alpha_{H}\int_0^T\int_0^t\exp(-\vartheta(t-s))\left((t-s)^{2H-2}-(t+s)^{2H-2}\right)dsdt}{\int_0^TX_t^2dt}+\frac{T}{2\int_0^TX_t^2dt}\,.
\end{equation}
\end{lem}
\begin{proof}
With the relationship of \eqref{eq:relation} we have
$$
\int_0^TX_t\delta \xi_t = \int_0^T X_t\circ d\xi_t-\alpha_{H}\int_0^T\int_0^t \exp\left(-\vartheta\left(t-s\right)\right)\left(\left(t-s\right)^{2H-2}-\left(t+s\right)^{2H-2}\right)dsdt-\frac{T}{2}.
$$
Together with the fact
$$
\int_0^T X_t\circ d\xi_t=\int_0^T X_t\circ dX_t+\vartheta\int_0^T X_t^2dt=\frac{X_T^2}{2}+\vartheta\int_0^T X_t^2dt.
$$
we can easily obtain the result.
\end{proof}
The strong consistency and asymptotical normality of the LSE will be presented with the following two theorems:
\begin{thm}\label{th: conv pr}
The LSE $\bar{\vartheta}_T$ defined in \eqref{eq3: least square} converges almost surely to $\vartheta$ as $T\rightarrow \infty$, that is 
$$
\lim_{T\rightarrow \infty}\bar{\vartheta}_T=\vartheta,\, a.s.\,.
$$
\end{thm}
The asymptotical laws of the LSE defined in \eqref{eq:least square} depends on the the Hurst parameter $H$ and we have the following results.

\begin{thm}\label{th: asym}
For $H\in (1/2,3/4)$ we have
\begin{equation} \label{varthetaclt1}
\sqrt{T}\left( \bar{\vartheta}_{T}-\vartheta \right) \xrightarrow{\mathcal {L}}%
\mathcal{N}\left( 0,\sigma_{H}^{2}\right) \,,
\end{equation}%
where $\sigma_{H}=\frac{\sqrt{\vartheta^{1-4H}H^{2}\left(4H-1\right)
\left(\Gamma\left(2H\right)^{2}+\frac{\Gamma(2H)\Gamma(3-4H)
\Gamma(4H-1)}{\Gamma(2-2H)}\right)+\frac{1}{2\vartheta}}}{\vartheta^{-2H} H\Gamma(2H)+\frac{1}{2\vartheta}}$.

For $H=3/4$, the LSE is also asymptotically normal with the convergence rate $\frac{\sqrt{T}}{\sqrt{\log (T)}}$, that is
\begin{equation}\label{varthetaclt2}
\frac{\sqrt{T}}{\sqrt{\log(T)}}\left(\bar{\vartheta}_T-\vartheta\right)\xrightarrow{\mathcal {L}}\mathcal{N}\left(0,\frac{9}{4\vartheta^2\left(\frac{3\sqrt{\pi}\vartheta^{-3/2}}{4}+\frac{1}{2}     \right)^2}\right).
\end{equation}

For $H>3/4$, we have
\begin{equation}\label{varthetaclt3}
T^{2-2H}\left(\bar{\vartheta}_T-\vartheta\right)\xrightarrow{\mathcal{L}}
-\frac{\vartheta^{-1}R_1}{\vartheta^{-2H}H\Gamma(2H)+\frac{1}{2\vartheta}}\,.
\end{equation}
where $R_1$ is the Rosenblatt random variables defined in Theorem 5.2 of \cite{HN17} and $\Gamma(\cdot)$ denotes the Gamma function.
\end{thm}
\begin{remark}
Although we have found the asymptotical properties for the LSE, but for different $H$ we have not a same convergence rate and variance. If we want to achieve this goal perhaps we can consider the one-step MLE with the local asymptotical property (LAN). We will leave it for the future research.
\end{remark}

\subsection{A practical estimator}

Though we have obtained some desired asymptotical properties of LSE, but when the LSE $\bar{\vartheta}_T$ depends on the unknown parameter $\vartheta$, it is still not possible to do the simulation. However, thanks to the conclusion \eqref{eq: for practical}
$$
\frac{1}{T}\int_0^TX_t^2dt \stackrel{\textit{a.s.}}{\longrightarrow}\frac{1}{2\vartheta}+H\vartheta^{-2H}\Gamma(2H).
$$
we can propose a practical estimator. In order to achieve this goal, let us define a function $p(\vartheta)=\frac{1}{2\vartheta}+H\vartheta^{-2H}\Gamma(2H)$. Then a practical estimator $\tilde{\vartheta}_T$ can be defined by
\begin{equation}\label{eq: practical estimator}
\tilde{\vartheta}_T=p^{-1}\left(\frac{1}{T}\int_0^T X_t^2\right).
\end{equation}

Obvious estimator $\tilde{\vartheta}_T$ converges to $\vartheta$ almost surely when $T\rightarrow \infty$. Moreover, we can obtain the asymptotical normality of $\tilde{\vartheta}_T$ with the Delta method. For the sake of saving space, we only present the case of $H\in (1/2,3/4)$ here and the other two cases ($H=3/4$ and $H \in (3/4,1)$) can also be obtained by the same method.
\begin{thm}\label{asym law ergodic}
As $T\rightarrow \infty$, when $H\in (1/2, 3/4)$
$$
\sqrt{T}\left( \tilde{\vartheta}_{T}-\vartheta \right) \xrightarrow{\mathcal {L}}%
\mathcal{N}\left( 0,\frac{\sigma_{H}^{2}\left(H\Gamma(2H)\vartheta^{1-2H}+\frac{1}{2}
\right)^{2}}{\vartheta^{2}}\right) \,,
$$
where $\sigma_{H}=\frac{\sqrt{\vartheta^{1-4H}H^{2}\left(4H-1\right)
\left(\Gamma\left(2H\right)^{2}+\frac{\Gamma(2H)\Gamma(3-4H)
\Gamma(4H-1)}{\Gamma(2-2H)}\right)+\frac{1}{2\vartheta}}}{\vartheta^{-2H} H\Gamma(2H)+\frac{1}{2\vartheta}}$
\end{thm}
\begin{proof}
The proofs are the same of \cite{HN10} with the equation \eqref{eq3: least square} 
\end{proof}
\begin{remark}
Here for the simulation, we have to use the function of $p^{-1}(\vartheta)$ but this is not an explicit function, the numerical result of the inverse function will be applied in MATLAB.
\end{remark}

\begin{remark}
Since the simulation friendly estimator, $\tilde{\vartheta}_T$, does not contain any stochastic integral and hence it is simpler to simulate. Motivated by Eq. (5.1) in \cite{HN17}, we choose to work with the formula \eqref{eq: practical estimator} by replacing the Riemann integral in the denominator by its corresponding approximate Riemann sums in discrete integer time. Specifically, we define,
\begin{equation}\label{practical estimator discrete}
\tilde{\vartheta}_{N}=p^{-1}\left(\frac{1}{N}\sum_{i=1}^{N} X_{id}^{2} \right)\,,
\end{equation}
where $d>0$ the sampling interval and the process $X_t$ is observed at discrete-time instants $t_{i}=id$, $i=1,2,\ldots,N$.
\end{remark}

\begin{remark}
Let $N\rightarrow\infty$, $d\rightarrow0$ and $H\in(\frac{1}{2},1)$. Borrowing the idea of \cite{HN17} and using Theorem \ref{asym law ergodic}, we can prove the strong consistency and the asymptotic laws for the practical estimator $\tilde{\vartheta}_{n}$ for under some mild conditions.
\end{remark}

\section{Non ergodic case}\label{sec: non ergodic}

When $\vartheta<0$, the process defined in \eqref{msfOU} is obviously non ergodic. However, even we can follow the same approach in \cite{el2016least}, the quadratic variation of $X_t$ is not 0, we can not consider the integral $\int_{0}^{T}X_{t}dX_{t}$ as the Young integral. As presented in the previous ergodic case, we still define it as the Skorohod integral then the following Lemma will  play the key role in this non ergodic situation:
\begin{lem}\label{lemma 4.1}
For $H>1/2$ and $\vartheta<0$, we have
$$
\lim_{T\rightarrow \infty}\frac{H(2H-1)\int_0^T \int_0^t \exp(-\vartheta(t-s))((t-s)^{2H-2}+(t+s)^{2H-2})dsdt}{\int_0^T X_t^2dt}=0
$$
and
$$
\lim_{T\rightarrow \infty}\frac{T}{\int_0^T X_t^2dt}=0.
$$
\end{lem}
\begin{proof}
Although for $0\leq t \leq T$, $c>0$ and $\gamma>0$, the condition of the msfBm $\xi_t: \mathbf{E}\xi_t^2\leq c t^{\gamma}$ is not satisfied, we can divide the process, $\xi_t$, into two parts: one is on the interval $[0,1]$ and the other is on interval $[1,\infty]$. For any interval, this condition is satisfied and the proof for Lemma 2.1 in \cite{el2016least} can be achieved by these two parts. Thus we obtain
\begin{equation}\label{eq32}
\lim_{T\rightarrow \infty} e^{2\vartheta T}\int_0^T X_t^2dt=-\frac{\vartheta}{2}Z_{\infty}^2,
\end{equation}
where $Z_t:=\int_0^t e^{\vartheta s}\xi_sds,\, t\geq 0$ and $Z_t\rightarrow Z_{\infty}$ almost surely in $L^2(\Omega)$.

Moreover, when $\vartheta<0$, it is easy to check that
\begin{equation}\label{eq33}
\lim_{T\rightarrow \infty}e^{2\vartheta T}T=0
\end{equation}
and
\begin{equation}\label{eq34}
\lim_{T\rightarrow \infty}e^{2\vartheta T}\int_0^T \int_0^t \exp(-\vartheta(t-s))((t-s)^{2H-2}+(t+s)^{2H-2})dsdt=0\,.
\end{equation}

Combining \eqref{eq32}, \eqref{eq33} with \eqref{eq34}, we obtain the desired
convergence. Thus we complete the proof.
\end{proof}

With this Lemma and from \eqref{eq3: least square}, we will just consider the estimator for $\vartheta<0$ as 
\begin{equation}\label{eq: estimator non ergodic}
\ddot{\vartheta}_T=-\frac{X_T^2}{2\int_0^T X_t^2dt}.
\end{equation}
Following similar steps as \cite{el2016least}, we can obtain the asymptotic consistency and asymptotic law of $\ddot{\vartheta}_T$.
\begin{thm}
Let $\vartheta<0$ and $H>1/2$. As $T\rightarrow \infty$, the estimator in \eqref{eq: estimator non ergodic} is strong consistency and asymptotical Cauchy
$$
e^{-\vartheta T}(\ddot{\vartheta}_T-\vartheta)  \stackrel{\mathcal{L}}{\longrightarrow} -\frac{2}{\vartheta}\mathcal{C}(1),
$$
where $\mathcal{C}(1)$ is a standard Cauchy distribution with the probability density function $\frac{1}{\pi(1+x^2)}$.
\end{thm}
\begin{proof}
The proof for this theorem is almost the same as in \cite{el2016least} which only needs to divide the independent part of sfBm and the standard Brownian motion.
\end{proof}

\section{Simulation study}\label{sec: simulation}

In Section \ref{preliminaries} we have introduced the method of the simulation of sub-fractional Brownian motion and in this part we will provide an algorithm for the estimating the drift parameter for msfOUP by Monte Carlo simulation:
\begin{itemize}

\item  Set $X_{0}=0$ and simulate the observations $X_{d},\ldots, X_{Nd}$ for different values of $H$ and $\vartheta$.  Here, we approximate the msfOUP by the Euler scheme:
\begin{equation}\label{xt simulation}
X_{(i+1)d}=X_{id}-\vartheta d X_{id}+\left(S_{(i+1)d}^{H}-S^{H}_{id}\right)+\left(W_{(i+1)d}-W_{id}\right), \quad i=0,\ldots, N.
\end{equation}
For each case, we simulate $l=10000$ paths.

\item  Obtain the practical estimator of (\ref{practical estimator discrete}), by solving the equation $\frac{1}{N}\sum_{i=1}^{N} X_{id}^{2}=\vartheta^{-2H}H\Gamma(2H)+\frac{1}{2\vartheta}$, numerically.

\end{itemize}

Now, setting $\vartheta=0.3$, $T=16$, $d=1/2^8$ and $X_0=0$, we simulate some paths of msfOUP with different Hurst parameters ($H=0.52, 0.62, 0.72$). The simulation paths reflex the main property of msfOUP: a large value of $H$ corresponds to a smoother sample path. In other words, for smaller values of $H$, the sample paths of a msfOUP fluctuate more wildly.

\begin{center}
\resizebox{160mm}{80mm}{\includegraphics{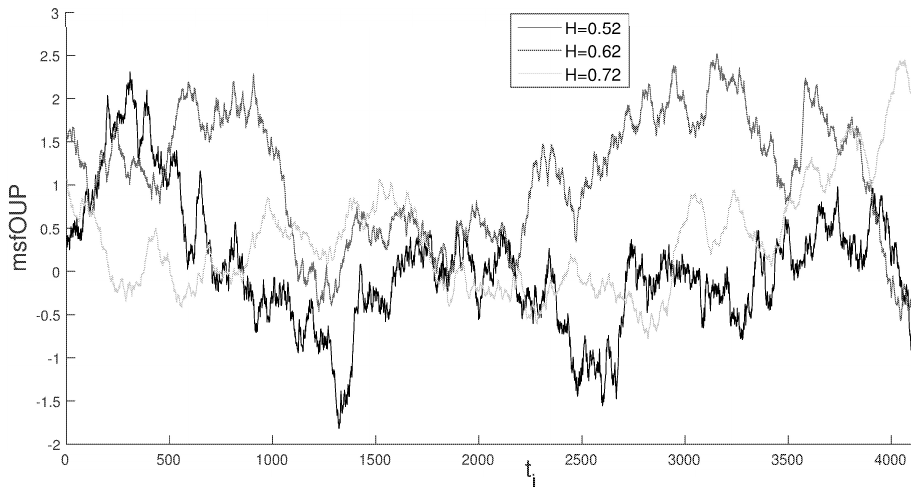}}\\
Fig.1. Generated msfOUP for different value of Hurst parameter.
\end{center}

In what follows, for some fixed sampling intervals $d = 1/12$ (e.g., data collected by monthly observations) and $d = 1/250$ (e.g., data collected by daily observations), we carry out a simulation study proposed above. Then, we obtain the practical estimator $\tilde{\vartheta}_{n}$ using some generating datasets with different sampling size $N$ and different time span $T$. For each case, replications involving $l=10000$
samples are simulated from the true model. The following table reports the mean, the median and standard deviation (S.Dev.) of the practical type estimator proposed by \eqref{practical estimator discrete} for different sample sizes and different time span, where the true values denote the parameter values used in the Monte Carlo simulation.

\begin{center}
\begin{tabular}{ccccccccccc}
\multicolumn{10}{c}{{\bf Table 1} Estimation results with the Hurst parameter $H=0.55$}\\
\toprule%
& True value & 0.1000 & 0.5000 & 1.0000 & 2.0000   & 0.1000 & 0.5000 & 1.0000 & 2.0000 \\
\hline
&& \multicolumn{4}{c}{$d=\frac{1}{12}$}&\multicolumn{4}{c}{$d=\frac{1}{250}$}\\
\cmidrule(l{1em}r){3-6} \cmidrule(l{1em}r){7-10}
\multirow{2}*{$T=10$}
& Mean      & 0.1461 & 0.6582 & 1.2895  & 2.4880   & 0.1223 & 0.5461 & 0.9743 & 2.3435 \\
& Median    & 0.1464 & 0.6783 & 1.3058  & 2.3888   & 0.1173 & 0.5341 & 0.9439 & 2.3639 \\
& S.Dev.    & 0.7544 & 0.8502 & 1.0040  & 0.7787   & 0.7014 & 0.8316 & 0.9814 & 0.7225 \\
\cmidrule(l{1em}r){3-6} \cmidrule(l{1em}r){7-10}
\multirow{2}*{$T=20$}
& Mean      & 0.1286 & 0.5835 & 1.1310  & 2.3596   & 0.1143 & 0.5149 & 1.0282 & 2.0672 \\
& Median    & 0.1324 & 0.5288 & 1.1559  & 2.3092   & 0.1277 & 0.5122 & 1.0448 & 2.0862\\
& S.Dev.    & 0.2358 & 0.3152 & 0.3910  & 0.5065   & 0.3678 & 0.4927 & 0.6155 & 0.5093 \\
\bottomrule
\end{tabular}
\end{center}

\vspace{5mm}
\begin{center}
\begin{tabular}{ccccccccccc}
\multicolumn{10}{c}{{\bf Table 2} Estimation results with the Hurst parameter $H=0.65$}\\
\toprule%
& True value & 0.1000 & 0.5000 & 1.0000 & 2.0000   & 0.1000 & 0.5000 & 1.0000 & 2.0000 \\
\hline
&& \multicolumn{4}{c}{$d=\frac{1}{12}$}&\multicolumn{4}{c}{$d=\frac{1}{250}$}\\
\cmidrule(l{1em}r){3-6} \cmidrule(l{1em}r){7-10}
\multirow{2}*{$T=10$}
& Mean      & 0.1492 & 0.6582 & 1.3698  & 2.6292   & 0.1114 & 0.5739 & 1.1608 & 2.2122 \\
& Median    & 0.1542 & 0.6783 & 1.3015  & 2.6721   & 0.1131 & 0.5647 & 1.1532 & 2.2314 \\
& S.Dev.    & 0.7461 & 0.8502 & 0.6455  & 0.8072   & 0.7182 & 0.7754 & 0.5010 & 0.5717 \\
\cmidrule(l{1em}r){3-6} \cmidrule(l{1em}r){7-10}
\multirow{2}*{$T=20$}
& Mean      & 0.1885 & 0.5844 & 1.1397  & 2.3824   & 0.1025 & 0.5153 & 1.0611 & 1.9846 \\
& Median    & 0.1304 & 0.5888 & 1.1435  & 2.4303   & 0.1044 & 0.5088 & 1.0739 & 2.0764\\
& S.Dev.    & 0.2412 & 0.3258 & 0.4069  & 0.5299   & 0.0991 & 0.1172 & 0.1285 & 0.1445 \\
\bottomrule
\end{tabular}
\end{center}

From numerical computations, we can see that the practical type estimator proposed in this paper performs well for the Hurst parameters $H>\frac{1}{2}$. As is expected, the simulated mean of these estimators converges to the true value rapidly and the simulated standard deviation decreases to zero with a slight positive bias as the sampling interval tends to zero and the time span goes to infinite.

To evidence the asymptotic laws of $\tilde{\vartheta}_{N}$, we next investigate the asymptotic distributions of $\tilde{\vartheta}_{N}$. Thus, we focus on the distributions of the following statistics:
\begin{equation}\label{function phi}
\Phi(N,H,\vartheta,d,X)=\frac{\vartheta\sqrt{Nd}}{\sigma_{H}\left(H\Gamma(2H)\vartheta^{1-2H}+\frac{1}{2}
\right)}\left(\tilde{\vartheta}_{N}-\vartheta\right)\,.
\end{equation}

Here, the chosen parameters are $\vartheta$=0.1, $H$=0.618 and we take $T$=16 and $h=\frac{1}{250}$.  We perform 10,000 Monte Carlo simulations of the sample paths generated by the process of \eqref{xt simulation}. The results are presented in the following Figure and Table 3.

\begin{center}
\resizebox{160mm}{80mm}{\includegraphics{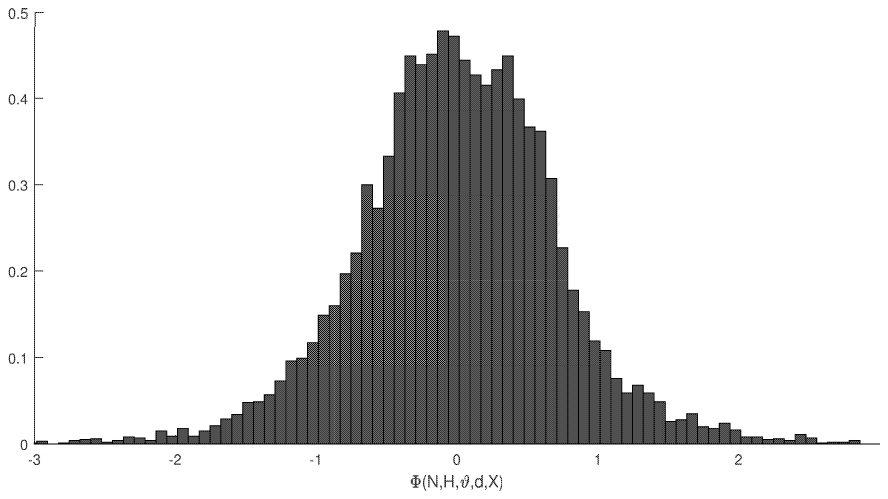}}\\
Fig.2. Histogram of the statistic $\Phi(N,H,\vartheta,d,X)$.
\end{center}

\begin{center}
\begin{tabular}{cccccccccccc}
\multicolumn{6}{c}{Table 3.  The comparisons of statistical
properties between $\Phi(N,H,\vartheta,d,X)$ and $\mathcal {N}$(0,1).}\\
\toprule%
Statistics  & Mean  & Median & Standard Deviation  & Skewness  & Kurtosis  \\
\hline
$\mathcal {N}$(0,1)            &  0          &0        & 1        & 0       &      3    \\
$\Phi(N,H,\vartheta,d,X)$ & 0.0003419 & 0.0716 & 0.000011648 &0.0246  &4.5534 \\
\bottomrule
\end{tabular}
\end{center}

The histogram indicates that the normal approximation of the distribution of the statistic $\Phi(N,H,\vartheta,d,X)$ is reasonable even when sampling size $N$ is not so large. From Table 3, we can see that the empirical mean, standard deviation, skewness and kurtosis are close to their asymptotic counterparts, which confirms our theoretical analysis: the convergence of the distribution of $\Phi(N,H,\vartheta,d,X)$ is fast. Thus, the density plot of the simulation results is close to the kernel of the limiting distribution of $\Phi(N,H,\vartheta,d,X)$ proposed by \eqref{function phi} when
$H =0.618$. For $H>\frac{3}{4}$, the limiting distribution, known as Rosenblatt distribution, is not known to have a closed form. Readers who are interested in the density plot of Rosenblatt random variable are referred to \cite{veillette2013properties} and the references therein.

\section{Appendix}

\subsection{Proof of Main Theorem}

In this part we will prove the main results of Theorem \ref{th: conv pr} and Theorem \ref{th: asym}. First of all, let us introduce the following Lemma.
\begin{lem}\label{cov distant}
Let $S_t^H$ be a sub-fractional Brownian motion. Then, we have
\begin{equation}\label{eq: add1}
\mathbf{E}\left[\int_s^T e^{-\vartheta(\xi-s)}dS_{\xi}^H\int_t^Te^{-\vartheta(\eta-t)}dS_{\eta}^H\right]\leq C_{\vartheta, H}|t-s|^{2H-2}
\end{equation}
and
\begin{equation}\label{eq: add2}
\mathbf{E}\left[\int_0^t e^{-\vartheta(t-u)}dS_u^H \int_0^s e^{-\vartheta(s-v)}dS_v^H\right]\leq C_{\vartheta, H}|t-s|^{2H-2}.
\end{equation}
\end{lem}
\begin{proof}
In fact
$$
\mathbf{E}\left[\int_s^T e^{-\vartheta(\xi-s)}dS_{\xi}^H\int_t^Te^{-\vartheta(\eta-t)}dS_{\eta}^H\right]=\alpha_H\int_t^T\int_s^T e^{-\vartheta(\xi-s)}e^{-\vartheta(\eta-t)}\left(|\xi-\eta|^{2H-2}-|\eta+\xi|^{2H-2}\right)d\xi d\eta
$$
when for fixed real number $t,s\geq 0$, we have $|t+s|^{2H-2}\leq |t-s|^{2H-2}$ we have 
$$
0<\mathbf{E}\left[\int_s^T e^{-\vartheta(\xi-s)}dS_{\xi}^H\int_t^Te^{-\vartheta(\eta-t)}dS_{\eta}^H\right]\leq 2 \alpha_H\int_t^T\int_s^T e^{-\vartheta(\xi-s)}e^{-\vartheta(\eta-t)}|\xi-\eta|^{2H-2}d\xi d\eta$$
From the web only Lemma 5.4 of \cite{HN10} we know
$$
 \alpha_H\int_t^T\int_s^T e^{-\vartheta(\xi-s)}e^{-\vartheta(\eta-t)}|\xi-\eta|^{2H-2}d\xi d\eta\leq C_{\vartheta, H}|t-s|^{2H-2}
 $$
 which achieves the proof of \eqref{eq: add1} and the same for \eqref{eq: add2}.
\end{proof}
The following Lemma plays key role in the proof of Theorem \ref{th: asym}  when in the norm $\|\cdot\|_{\mathcal{H}}$ we have to calculate the inner product with respect to the standard Brownian motion.
\begin{lem}\label{s s distance}
For $H>1/2$, we have
\begin{equation}\label{eq: s s distance}
\lim_{T\rightarrow \infty}\frac{1}{T}\int_0^T \mathbf{E}\left(\int_0^s e^{-\vartheta(s-u)}dB_u^H\int_s^T e^{-\vartheta(v-s)}dB_v^H      \right)ds=\vartheta^{-2H}\Gamma(2H)\,.
\end{equation}
\end{lem}
\begin{proof}
\begin{eqnarray*}
\int_0^T \mathbf{E}\left(\int_0^s e^{-\vartheta(s-u)}dB_u^H\int_s^T e^{-\vartheta(v-s)}dB_v^H      \right)ds&=& \int_0^T\int_s^T\int_0^s e^{-\vartheta(y-x)}(y-x)^{2H-2}dxdyds\\
&=&\int_0^T\int_0^y\int_x^y e^{-\vartheta(y-x)}(y-x)^{2H-2}dsdxdy\\
&=&\int_0^T\int_0^ye^{-\vartheta (y-x)}(y-x)^{2H-1}dxdy\\
&=&\int_0^T\int_0^ye^{-\vartheta x}x^{2H-1}dxdy=\int_0^T\int_x^T e^{-\vartheta x}x^{2H-1}dydx\\
&=&\int_0^Te^{-\vartheta x}x^{2H-1}(T-x)dx\\
&=&T\int_0^T e^{-\vartheta x} x^{2H-1}dx-\int_0^Te^{-\vartheta x}x^{2H}dx
\end{eqnarray*}
The two limits 
$$
\int_0^{\infty}e^{-\vartheta x}x^{2H-1}dx=\vartheta^{-2H}\Gamma(2H),\,\,\,\,\,\, \lim_{T\rightarrow \infty}\frac{1}{T}\int_0^T e^{-\vartheta x}x^{2H}dx=0
$$
complete the proof.

\end{proof}

\subsubsection{Proof of Theorem \ref{th: conv pr}}

The result in \cite{HN10} gives the strong consistency for the LSE from the ergodicity. Since the increment of the sfBm is not stationary, we can not use the ergodicity to prove the consistency of $\bar{\vartheta}_T$. Now, a standard calculation yields
\begin{eqnarray*}
&&\lim_{T\rightarrow \infty} \frac{H(2H-1)}{T}\int_0^T\int_0^t\exp(-\vartheta(t-s))(t-s)^{2H-2}dsdt \\ &=&\lim_{T\rightarrow \infty}\frac{H(2H-1)}{T}\int_0^T\int_0^t u^{2H-2}e^{-\vartheta u}dudt\\
&=&\vartheta^{1-2H}H\Gamma(2H).
\end{eqnarray*}

On the other hand, a straightforward calculation shows that
\begin{eqnarray*}
\int_0^T\int_0^t \exp(-\vartheta(t-s))(t+s)^{2H-2}ds&=&\int_0^T\int_0^t \exp (\vartheta(s+t)-2\vartheta t)(t+s)^{2H-2}dsdt\\
&=&\int_0^T\exp(-2\vartheta t)\int_t^{2t}\exp(\vartheta u)u^{2H-2}dudt\,.
\end{eqnarray*}

With the L'H\^{o}spital's rule, we have
$$\lim_{T\rightarrow \infty}\frac{1}{T}\int_0^T\exp(-2\vartheta t)\int_t^{2t}\exp(\vartheta u)u^{2H-2}dudt=0.$$

Now we only need to prove that
\begin{equation}\label{eq: for practical}
\frac{1}{T}\int_0^TX_t^2dt \stackrel{\textit{a.s.}}{\longrightarrow}\frac{1}{2\vartheta}+H\vartheta^{-2H}\Gamma(2H).
\end{equation}

For $0\leq t\leq T$, let $W_t$ be a standard Brownian motion and $X_t=X_t^{(1)}+X_t^{(2)}+X_t^{(3)}$. Then, we have
$$
dX_t^{(1)}=-\vartheta X_t^{(1)}dt+dW_t,\, 0\leq t \leq T,
$$
and
$$
dX_t^{(2)}=-\vartheta X_t^{(2)}dt+\frac{1}{\sqrt{2}}dB_t^H,\, 0\leq t\leq T,
$$
$$
dX_t^{(3)}=-\vartheta X_t^{(3)}dt+\frac{1}{\sqrt{2}}dB_{-t}^H,\, 0\leq t \leq T.
$$
With the ergodic property of these three processes presented in \cite{CKM03}, we have the following results:
$$
\frac{1}{T}\int_0^T\left(X_t^{(1)}\right)^2dt \stackrel{a.s.}{\longrightarrow}\frac{1}{2\vartheta},
$$
$$
\frac{1}{T}\int_0^T\left(X_t^{(2)}\right)^2dt \stackrel{a.s.}{\longrightarrow}\frac{H}{2}\vartheta^{-2H}\Gamma(2H)
$$
and
$$
\frac{1}{T}\int_0^T\left(X_t^{(3)}\right)^2dt \stackrel{a.s.}{\longrightarrow}\frac{H}{2}\vartheta^{-2H}\Gamma(2H)
$$
so \eqref{eq: for practical} needs this result:
\begin{equation}\label{eq: 1 and 2 as}
\frac{1}{T}\int_0^T X_t^{(2)} X_t^{(3)}dt \stackrel{a.s.}{\longrightarrow} 0.
\end{equation}
First of all 
\begin{equation}
\mathbf{E}X_1^{(1)}X_t^{(2)}= \int_0^t \int_0^t e^{-\vartheta (t-s)}e^{-\vartheta (t-u)}|u+s|^{2H-2}duds
\end{equation}
when $|u+s|^{2H-2}\leq (us)^{H-1}$ for $u,s\geq 0$, then with L'H\^opital rule we have the following inequality:
\begin{equation}\label{eq: control 1 and 2}
\mathbf{E}X_1^{(1)}X_t^{(2)}\leq C_{\vartheta, H}\left(1 \wedge t^{2H-2} \right). 
\end{equation}
With this inequality, we can easily obtain the convergence in probability 
$$
\frac{1}{T}\int_0^T X_t^{(2)} X_t^{(3)}dt \stackrel{\mathbf{P}}{\longrightarrow} 0.
$$
From \eqref{eq: control 1 and 2}, Borel-cantelli Lemma we can easily obtain the convergence almost surely in \eqref{eq: 1 and 2 as} for $1/2<H<3/4$. For the case $3/4\leq H<1$, we will apply the method of The Theorem 2.1 in \cite{CHW17}. In fact with the convergence in probability means there exists an sub-sequence which convergents almost surely to 0, on the other hand,  the equation \eqref{eq: control 1 and 2} verifies the condition of second order Winer-Ito chaos in Proposition 3.4 of \cite{CHW17}, together with GRR inequality (See Theorem 2.1 in \cite{Hu17}) \eqref{eq: 1 and 2 as} will be achieved.

\subsubsection{Proof of Theorem \ref{th: asym}}

{\bf Step 1:} We shall use Malliavin calculus and the fourth moment theorem (see, for example, Theorem 4 in \cite{NO2008}) to prove (\ref{varthetaclt1}).

In fact, using (\ref{eq:least square}), we have
\begin{equation}\label{eq: pre LSE}
\sqrt{T}\left(\bar{\vartheta}_T-\vartheta\right)=-\frac{\frac{1}{\sqrt{T}}\int_0^T\left(\int_0^t e^{-\vartheta(t-s)}d\xi_s^H\right)d\xi_t^H}{\frac{1}{T}\int_0^T X_t^2dt}=\frac{-F_T}{\frac{1}{T}\int_0^T X_t^2dt}\,,
\end{equation}
where $F_T$ is the double stochastic integral
\begin{equation}\label{eq: dou int}
F_T=\frac{1}{2\sqrt{T}}I_2\left(e^{-\vartheta|t-s|}\right)=\frac{1}{2\sqrt{T}}\int_0^T\int_0^T e^{-\vartheta |t-s|}d\xi_s\xi_t.
\end{equation}

By \eqref{eq: for practical}, we know that $\frac{1}{T}\int_0^T X_t^2dt$ converges in probability and in $L^2$ as $T$ tends to infinity to $\frac{1}{2\vartheta}+H\vartheta^{-2H}\Gamma(2H)$. From Theorem 4 of \cite{NO2008}, we have to check the following two conditions:

\begin{description}
  \item[(i).] $\mathbf{E}(F_T^2)$ converges to a constant as T tends to infinity
$$
\lim_{T\rightarrow \infty} \mathbf{E}F_T^2=\vartheta^{1-4H}H^2(4H-1)\left(\Gamma(2H)^2
+\frac{\Gamma(2H)\Gamma(3-4H)\Gamma(4H-1)}{\Gamma(2-2H)}\right)+\frac{1}{2\vartheta}\,.
$$
  \item[(ii).] $\|DF_T\|_{\mathcal{H}}^2$ converges in $L^2$ to a constant as $T$ tends to infinity.
\end{description}

We first check the condition {\bf (i)}. When $W_t$ and $S_t^H$ are independent, we have $\mathbf{E}F_T^2=\mathbf{E}\left(F_{1,T}^2+F_{T,2}^2\right)$, with
$F_{1,T}=\frac{1}{2\sqrt{T}}\int_0^T\int_0^T e^{-\vartheta |t-s|}dW_tdW_s,\,\, F_{2,T}=\frac{1}{2\sqrt{T}}\int_0^T\int_0^T e^{-\vartheta |t-s|}dS_t^HdS_s^H.$

A standard calculation together with (\ref{fg}) yields
\begin{eqnarray*}
\mathbf{E}F_{2,T}^2&=&\frac{\alpha_H^2}{2T}\int_{[0,T]^4}\exp\left(-\vartheta |u_2-s_2|-\vartheta|u_1-s_1|\right)|u_2-u_1|^{2H-2}|s_2-s_1|^{2H-2}ds_1ds_2du_1du_2\\
&&-\frac{\alpha_H^2}{2T}\int_{[0,T]^4}\exp\left(-\vartheta |u_2-s_2|-\vartheta|u_1-s_1|\right)|u_2-u_1|^{2H-2}|s_2+s_1|^{2H-2}ds_1ds_2du_1du_2\\
&&-\frac{\alpha_H^2}{2T}\int_{[0,T]^4}\exp\left(-\vartheta |u_2-s_2|-\vartheta|u_1-s_1|\right)|u_2+u_1|^{2H-2}|s_2-s_1|^{2H-2}ds_1ds_2du_1du_2\\
&&+\frac{\alpha_H^2}{2T}\int_{[0,T]^4}\exp\left(-\vartheta |u_2-s_2|-\vartheta|u_1-s_1|\right)|u_2+u_1|^{2H-2}|s_2+s_1|^{2H-2}ds_1ds_2du_1du_2\,.
\end{eqnarray*}

From \cite{HN10}, we have
\begin{eqnarray}\label{ef2t first}
&&\lim_{T\rightarrow \infty}\frac{\alpha_H^2}{2T}\int_{[0,T]^4}\exp\left(-\vartheta |u_2-s_2|-\vartheta|u_1-s_1|\right)|u_2-u_1|^{2H-2}|s_2-s_1|^{2H-2}ds_1ds_2du_1du_2 \notag\\
&=&\vartheta^{1-4H}H^2(4H-1)\left(\Gamma(2H)^2+\frac{\Gamma(2H)\Gamma(3-4H)\Gamma(4H-1)}{\Gamma(2-2H)}\right).
\end{eqnarray}

A simple calculation yields
\begin{eqnarray}\label{ef1t}
\lim_{T\rightarrow \infty}\mathbf{E}F_{1,T}^2=\lim_{T\rightarrow \infty}\frac{1}{T}\int_0^T\int_0^te^{-2\vartheta(t-s)}dsdt=\frac{1}{2\vartheta}.
\end{eqnarray}

Now, if we can prove
\begin{equation}\label{eq: Cai}
\lim_{T\rightarrow \infty}\frac{1}{T}\int_{[0,T]^4}\exp\left(-\vartheta |u_2-s_2|-\vartheta|u_1-s_1|\right)|u_2-u_1|^{2H-2}|s_2+s_1|^{2H-2}ds_1ds_2du_1du_2=0\,,
\end{equation}
then the last three terms of $\mathbf{E}F_{2,T}^2$ will tend to zero with the fact $|s_2+s_1|^{2H-2}\leq |s_2-s_1|^{2H-2}$.

Denote
$$
I_T=\frac{1}{T}\int_{[0,T]^4}e^{-\vartheta |u_2-s_2|-\vartheta|u_1-s_1|}\left(|u_2-u_1|^{2H-2}|s_2+s_1|^{2H-2}\right)ds_1ds_2du_1du_2\,.
$$

Using the L'H\^{o}spital's rule, we have
$$
\frac{dI_T}{dT}=\int_{[0,T]^3}e^{-\vartheta (T-s_2)-\vartheta|s_1-u_1|}\left((T-u_1)^{2H-2}(s_2+s_1)^{2H-2}\right)ds_1du_1du_2.
$$

Let $T-s_2=x_1, T-s_1=x_2, T-u_1=x_3$. Ignoring the sign, we have
\begin{eqnarray*}
\frac{dI_T}{dT}&=&\int_{[0,T]^3}e^{-\vartheta x_1-\vartheta |x_2-x_3|}\left(x_3^{2H-2}(T-x_1+T-x_2)^{2H-2}\right)dx_1dx_2dx_3\\
&=& e^{-\vartheta T}\int_{[0,T]^3}e^{\vartheta y_1-\vartheta|x_2-x_3|}\left(x_3^{2H-2}(y_1+T-x_2)^{2H-2}\right)dy_1dx_2dx_3\\
&\leq& e^{-\vartheta T}\int_{[0,T]^3}e^{\vartheta y_1-\vartheta|x_2-x_3|}\left(x_3^{2H-2}y_1^{2H-2}\right)dy_1dx_2dx_3.
\end{eqnarray*}

Let $ J_T=\int_{[0,T]^3}e^{\vartheta y_1-\vartheta|x_2-x_3|}\left(x_3^{2H-2}y_1^{2H-2}\right)dy_1dx_2dx_3$. Then, using the L'H\^{o}spital's rule, we get
$$
\frac{dJ_T}{dT}=e^{-\vartheta T}\int_{[0,T]^2}e^{\vartheta y_1+\vartheta x_3}(y_1^{2H-2}x_3^{2H-2})dy_1dx_3\,.
$$

On the other hand, we can easily obtain $\frac{de^{-\vartheta T}}{dT}=-\vartheta e^{-\vartheta T}$. Moreover, with the L'H\^{o}spital's rule, it is easy to check that
$$
\lim_{T\rightarrow \infty} \frac{J_T}{e^{\vartheta T}}=0\,,
$$
which implies the equation \eqref{eq: Cai}.

Consequently, with \eqref{eq: Cai} and the fact $|s_2+s_1|^{2H-2}\leq |s_2-s_1|^{2H-2}$, it is easy to see that
\begin{equation}\label{eq44}
\lim_{T\rightarrow \infty}\frac{1}{T}\int_{[0,T]^4}\exp\left(-\vartheta |u_2-s_2|-\vartheta|u_1-s_1|\right)|u_2+u_1|^{2H-2}|s_2+s_1|^{2H-2}ds_1ds_2du_1du_2=0.
\end{equation}

Combining \eqref{ef2t first}, \eqref{ef1t}, \eqref{eq: Cai} with \eqref{eq44}, we verify condition {\bf (i)}.

Now we will check the condition condition {\bf (ii)}. For $s\leq T$, we have
$$
D_sF_T=\frac{X_s}{\sqrt{T}}+\frac{1}{\sqrt{T}}\int_s^T e^{-\vartheta(t-s)}d\xi_t\,.
$$

From \eqref{fg} we have
\begin{eqnarray*}
\|D_sF_T\|_{\mathcal{H}}^2&=&\frac{1}{T}\int_0^T \left(X_s+\int_s^T e^{-\vartheta(t-s)}d\xi_t\right)^2ds+\frac{H(2H-1)}{T}\\
&&\int_0^T\int_0^T\left(X_s+\int_s^Te^{-\vartheta(t-s)} d\xi_t\right)\left(X_u+\int_u^Te^{-\vartheta(t-u)}d\xi_t\right)\left(|u-s|^{2H-2}-|u+s|^{2H-2}\right)duds
\end{eqnarray*}

We first consider the first term of the above equation. A straightforward calculation shows that
\begin{eqnarray*}
\frac{1}{T}\int_0^T \left(X_s+\int_s^T e^{-\vartheta(t-s)}d\xi_t\right)^2ds&=&\frac{1}{T}\int_0^T \left(X_s^2+2X_s\int_s^Te^{-\vartheta(t-s)}d\xi_t+\left(\int_s^T e^{-\vartheta(t-s)}d\xi_t \right)^2             \right)ds\\
&=&A_T^{(1)}+A_T^{(2)}+A_T^{(3)}.
\end{eqnarray*}

From the proof of Theorem \ref{th: conv pr}, the independent of the $W_t$ and $S_t^H$ in the msfBm, the convergence to $0$ for the standard Brownian motion case in the proof of Theorem 3.4 of \cite{HN10}, the ergodicity and stationary of the fractional O-U process (see \cite{CKM03}) and Lemma \ref{s s distance} we can easily obtain that all these three terms converges in $L^2$ as $T$ tends to infinity.

Now let us look at the third term of $\|D_sF_T\|_{\mathcal{H}}^2$. A standard calculation yields
\begin{eqnarray*}
C_T&=&\frac{H(2H-1)}{T}\int_0^T\int_0^T\left(X_s+\int_s^Te^{-\vartheta(t-s)} d\xi_t\right)\left(X_u+\int_u^Te^{-\vartheta(t-u)}d\xi_t\right)\\
&&\;\left(|u-s|^{2H-2}-|u+s|^{2H-2}\right)duds\\
&=&\frac{H(2H-1)}{T}\left(C_T^{(1)}+2C_T^{(2)}+C_T^{(3)}\right)\,,
\end{eqnarray*}
where
$$
C_T^{(1)}=\int_0^T\int_0^TX_sX_u\left(|u-s|^{2H-2}-|u+s|^{2H-2}\right)duds,
$$
$$
C_T^{(2)}=\int_0^T\int_0^T\left(X_u \int_s^Te^{-\vartheta(t-s)} d\xi_t\right) \left(|u-s|^{2H-2}-|u+s|^{2H-2}\right) duds\,,
$$
and
$$
C_T^{(3)}=\int_0^T\int_0^T\left( \int_s^Te^{-\vartheta(t-s)} d\xi_t  \int_u^Te^{-\vartheta(t-u)}d\xi_t          \right)  \left(|u-s|^{2H-2}-|u+s|^{2H-2}\right) duds.
$$

With the same method of Theorem 3.4 in \cite{HN10}, Lemma \ref{cov distant} and the independence of the $W_t$ and $S_t^H$ in msfBm, we have 
\begin{eqnarray}\label{eq48}
\lim_{T\rightarrow \infty}\frac{1}{T^2}\mathbf{E}\left(|C_T^{(i)}-\mathbf{E}C_T^{(i)}|^2\right)&=&0,\,\,\,\,i=2,3\,.
\end{eqnarray}
then we have 
\begin{equation}\label{conv 0 C T}
\lim_{T\rightarrow \infty} \mathbf{E}\left[\left(C_T-\mathbf{E}C_T\right)^2\right]=0.
\end{equation}
which implies that $\lim\limits_{T\rightarrow\infty}\mathbf{E} C_T$ exists. Finally, we obtain that $C_T$ converges in $L^2$ to a constant. Thus, condition {\bf (ii)} satisfies.

{\bf Step 2:} Case $H=3/4$. From \eqref{eq:least square} we have
$$
\frac{\sqrt{T}}{\sqrt{\log (T)}}\left(\bar{\vartheta}_T-\vartheta\right)=-\frac{\frac{F_T}{\sqrt{\log T}}}{\frac{1}{T}\int_0^TX_t^2dt}\,,
$$
where $F_T$ is defined by \eqref{eq: dou int}.

We still use the fourth moment theorem (see, for example, Theorem 4 in \cite{NO2008}) and check two conditions of {\bf Step 1}. Using same calculations of {\bf Step 1}, we can  show that
$$
\lim_{T\rightarrow \infty}\frac{1}{T\log(T)}\int_{[0,T]^4}e^{-\vartheta|s_2-u_2|-\vartheta|s_1-u_1|}|s_2-s_1|^{2H-2}(u_2+u_1)^{2H-2}du_1du_2ds_1ds_2=0
$$
and
$$
\lim_{T\rightarrow \infty}\frac{1}{T\log(T)}\int_{[0,T]^4}e^{-\vartheta|s_2-u_2|-\vartheta|s_1-u_1|}(s_2+s_1)^{2H-2}(u_2+u_1)^{2H-2}du_1du_2ds_1ds_2=0.
$$

On the other hand, a straightforward calculation shows that
$$
\lim_{T\rightarrow \infty}\mathbf{E}\left(\frac{1}{2\sqrt{T\log T}}\int_0^T\int_0^T e^{-\vartheta|t-s|}dW_tdW_s\right)^2=0\,.
$$

Then, we have
\begin{eqnarray*}
&&\lim_{T\rightarrow \infty}\mathbf{E}\left(\frac{F_T}{\sqrt{\log T}}\right)^2\\
&=&\lim_{T\rightarrow \infty}\frac{H^2(2H-1)^2}{2T\log(T)}\int_{[0,T]^4}e^{-\vartheta|s_2-u_2|-\vartheta|s_1-u_1|}|s_2-s_1|^{2H-2}(u_2-u_1)^{2H-2}du_1du_2ds_1ds_2\\
&=&\frac{9}{4\vartheta^2}\,,
\end{eqnarray*}
where the equality comes from Lemma 6.6 in \cite{HN17}.

Thus, condition {\bf (i)} and condition {\bf (ii)} are obvious when we add a term of $\frac{1}{\sqrt{\log T}}$ and $T^{8H-6}=1$ with $H=\frac{3}{4}$.

{\bf Step 3:} In this step we will prove the theorem when $3/4<H<1$. From \eqref{eq:least square}, we have
$$
T^{2-2H}\left(\bar{\vartheta}_T-\vartheta\right)=-\frac{\frac{T^{1-2H}}{2}\int_0^T\int_0^Te^{-\vartheta|t-s|}d\xi_sd\xi_t}{\frac{1}{T}\int_0^TX_t^2dt}.
$$

Let us mention that the condition {\bf (ii)} in {\bf Step 1} will not be satisfied when $H>3/4$. Fortunately, we still have the following convergence:
$$
\lim_{T\rightarrow \infty}T^{3-4H}\frac{1}{T}\int_{[0,T]^4}e^{-\vartheta|s_2-u_2|-\vartheta|s_1-u_1|}(s_2-s_1)^{2H-2}(u_2+u_1)^{2H-2}du_1du_2ds_1ds_2=0
$$
and
$$
\lim_{T\rightarrow \infty}T^{3-4H}\frac{1}{T}\int_{[0,T]^4}e^{-\vartheta|s_2-u_2|-\vartheta|s_1-u_1|}(s_2+s_1)^{2H-2}(u_2+u_1)^{2H-2}du_1du_2ds_1ds_2=0.
$$
With the similarity of the process $\xi$ and Lemma 6.6 in \cite{HN17}, we have
$$
{T^{1-2H}}\int_0^T\int_0^Te^{-\vartheta|t-s|}d\xi_sd\xi_t\xrightarrow{\mathcal{L}}2\vartheta^{-1}R_1
$$
which achieves the proof.

\vspace{1mm}

\textbf{Funding}: Chunhao Cai is supported by the Fundamental Research Funds for the SUFE No. 2020110294. Weilin Xiao is supported by the National Natural Science Foundation of China, grant No. 71871202.


\begin{thebibliography}{00}

\bibitem{FHBZ} Biagini, F.,  Hu,  Y., \O ksendal, B.,  Zhang, T. (2008) {\it Stochastic calculus for fractional Brownian motion and application}, Springer.

\bibitem{CCK} Cai, C., Chigansky, P., Kleptsyna, M. (2016) {\it Mixed Gaussian process: a filtering approach}, {Annals of probability}, 44(4), 3032-3075.

 
\bibitem{CK18a} Chigansky, P., Kleptsyna, M. (2018) {\it Exact asymptotics in eigenproblems for fractional covariance operators}, {Stochastic processes and their applications}, 128(6), 2007-2059.

\bibitem{CK18b} Chigansky, P., Kleptsyna,  M. (2018) {\it Statistical analysis of mixed fractional Ornstein-Uhlenbeck process}, {Theory of probability and its application}, 63(3), 500-519.
    


\bibitem{el2016least}  El. Machkouri, M., Es-Sebaiy, K., Ouknine, Y., (2016) {\it Least squares estimator for non-ergodicOrnstein–Uhlenbeck processes driven by gaussian processes}, J. Korean. Stat. Soc. 45, 329–341.   

  
\bibitem{HN10} Hu, Y., Nualart, D.  (2010) {\it Parameter estimation for fractional Ornstein-Uhlenbeck processes}, {Statistics and Probability Letters}, 80, 1030-1038.

\bibitem{HN17} Hu, Y., Nualart, D., Zhou, H.  (2019) {\it Parameter estimation for fractional Ornstein-Uhlenbeck processes of general Hurst parameter}.{Statistical Inference for Stochastic Processes}, 22, 111-142.
 

\bibitem{Klep02} Kleptsyna, M., Le Breton, A. (2002) {\it Statistical analysis of the fractional Ornstein-Uhlenbeck type process}, Statistical Inference for Stochastic Processes, 5 (3), 229-248, 2002    


\bibitem{CHW17} Chen, Y.,  Hu, Y., Wang, Z.(2017) {\it Parameter estimation of complex fractional Ornstein-Uhlenbeck processes with fractional noise}, ALEA, Lat. Am. J. Probab. Math. Stat., 14, 613-629.

\bibitem{Hu17} Hu, Y. (2017) {\it Analysis on Gaussian process}, World Scientific Publishing Co. Pte. Ltd., Hackensack, NJ.


\bibitem{zili13} Zili, M. (2013) {\it Mixed sub-fractional Brownian motion}, Rondom Operators and Stochastic Equation, 22(3).


\bibitem{CKM03} Cheridito, P.,  Kawaguchi, H.,   Maejima, M.  (2003) {\it Fractional Ornstein-Uhlenbeck processes}, Electron Journal of Probability, 8, 1-14.

\bibitem{GJR18} Gatheral, J., Jaisson, T., Rosenbaum, M. (2018) {\it Volatility is rough}, Quantitative Finance, 18(6), 933-949.


\bibitem{EFR18} Euch, O., Fukasawa, M., Rosenbaum, M. (2018) {\it The microstructural foundations of leverage effect and rough volatility}, Finance and Stochastics, 22(2), 241-280.


\bibitem{ER18} Euch, O., Rosenbaum, M. (2018) {\it Perfect hedging under rough Heston models}, Annals of Applied Probability, 28(6), 3813-3856.

\bibitem{ER19} Euch, O., Rosenbaum, M. (2019) {\it The characteristic function of rough Heston models}, Mathematical Finance, 29(1), 3-38, 2019.

\bibitem{MF15} Morozewicz, A.,Filatova,  D.  (2015) {\it On the simulation of sub-fractional Brownian motion}, {20th international conference on methods and models in automation and robotics}.


\bibitem{Nualart} Nualart, D.  (2006)  {\it The malliavin calculus and related topics}, Second edition, Springer.

\bibitem{NO2008} Nualart, D., Ortiz-Latorre, S.  (2008). {\it Central limit theorems for multiple stochastic integrals and Malliavin calculus}. Stochastic Process. Appl. 118, 614-628.

\bibitem{paxson1997} Paxson,  V.  (1997) {\it Fast, approximate synthesis of fractional Gaussian noise for generating self-similar network traffic}. ACM SIGCOMM Computer Communications Review, 27(5), 5-18.

\bibitem{liu2010sub} Liu, J., Li, L., Yan, L. (2010). Sub-fractional model for credit risk pricing. International Journal of Nonlinear Sciences and Numerical Simulation, 11(4):231-236.

\bibitem{veillette2013properties} Veillette, M.,  Taqqu, M.  (2013). {\it Properties and numerical evaluation of the Rosenblatt distribution}. Bernoulli, 19(3):982-1005
\end{thebibliography}
\end{document}